\documentclass [11pt,reqno]{amsart}
\usepackage {amsmath,amssymb,epsfig,verbatim,geometry}
\usepackage[all]{xy}
\usepackage[pagebackref,pdftex,hyperindex]{hyperref}

\newcommand{\A}{\mathbf{A}}
\newcommand{\C}{\mathbf{C}}

\newcommand{\Q}{\mathbf{Q}}
\newcommand{\R}{\mathbf{R}}
\newcommand{\Z}{\mathbf{Z}}
\newcommand{\N}{\mathbf{N}}
\renewcommand{\P}{\mathbf{P}}

\newcommand{\fa}{\mathfrak{a}}
\newcommand{\fb}{\mathfrak{b}}
\newcommand{\fm}{\mathfrak{m}}

\newcommand{\cE}{\mathcal{E}}

\newcommand{\cO}{\mathcal{O}}

\renewcommand{\a}{\alpha}
\renewcommand{\b}{\beta}
\renewcommand{\d}{\delta}
\newcommand{\e}{\varepsilon}
\newcommand{\f}{\varphi}

\newcommand{\D}{\Delta}

\newcommand{\eg}{{e.g.\ }}
\newcommand{\ie}{{i.e.\ }}

\renewcommand{\=}{:=}

\DeclareMathOperator{\Conv}{Conv}

\DeclareMathOperator{\Div}{Div}

\DeclareMathOperator{\vol}{Vol}

\DeclareMathOperator{\diam}{diam}

\DeclareMathOperator{\lip}{Lip}
\DeclareMathOperator{\sta}{Star}

\DeclareMathOperator{\Pic}{Pic}

\DeclareMathOperator{\ord}{ord}

\DeclareMathOperator{\Nw}{Nw}

\DeclareMathOperator{\val}{val}
\DeclareMathOperator{\Val}{Val}

\numberwithin{equation}{section}       
\newtheorem{prop} {Proposition} [section]
\newtheorem{thm}[prop] {Theorem}
\newtheorem{defi}[prop] {Definition}
\newtheorem{lem}[prop] {Lemma}

\newtheorem{prop-def}[prop]{Proposition-Definition}

\newtheorem*{thmA}{Theorem A}
\newtheorem*{thmAp}{Theorem A'}
\newtheorem*{thmB}{Theorem B}

\newtheorem*{IzumisThm}{Izumi's Theorem}

\newtheorem*{defi*} {Definition}

\newtheorem*{corC}{Corollary C}
\newtheorem*{corD}{Corollary D}
\newtheorem*{corE}{Corollary E}

\newtheorem{rmk}[prop]{Remark}
\theoremstyle{remark}

\title[Izumi's Theorem]{A refinement of Izumi's Theorem}
\date{\today}

\author{S{\'e}bastien Boucksom
  \and
  Charles Favre
  \and
  Mattias Jonsson}

\address{CNRS-Universit{\'e} Paris 6\\
  Institut de Math{\'e}matiques\\
  F-75251 Paris Cedex 05\\
  France}
\email{boucksom@math.jussieu.fr}

\address{CNRS-CMLS\\ \'Ecole Polytechnique\\91128 Palaiseau Cedex France}
\email{favre@math.polytechnique.fr}

\address{Dept of Mathematics\\
  University of Michigan\\
  Ann Arbor, MI 48109-1043\\
  USA}
\email{mattiasj@umich.edu}

\thanks{Second author supported by the ANR-project BERKO\@. 
Third author supported by the CNRS and the NSF}

\begin{document}

\begin{abstract}
  We improve Izumi's inequality, which states that 
  any divisorial valuation $v$ 
  centered at a closed point $0$ on an algebraic variety $Y$
  is controlled by the order of vanishing at $0$.
  More precisely, as $v$ ranges through valuations that are
  monomial with respect to coordinates in a fixed birational 
  model $X$ dominating $Y$, we show that for any 
  regular function $f$ on $Y$ at $0$, the function
  $v\mapsto v(f)/\ord_0(f)$ is uniformly Lipschitz continuous 
  as a function of the weight defining $v$.
  As a consequence, the volume of $v$ is also a 
  Lipschitz continuous function.
  Our proof uses toroidal techniques as well as positivity 
  properties of the images of suitable nef divisors under birational morphisms.
\end{abstract}

\maketitle

\setcounter{tocdepth}{1}
\tableofcontents

%
%

\section*{Introduction}
Let $Y$ be a normal variety over an algebraically 
closed field $k$
and let $0\in Y$ be a closed point. 
Write $\fm_0$ for the maximal ideal of
the local ring $\cO_{Y,0}$ at $0$ and define 
$\ord_0:\cO_{Y,0}\to\Z_{\ge0}\cup\{\infty\}$, the order of vanishing at 0, by  
\begin{equation*}
  \ord_0(f):=\max\{j\ge 0\mid f\in\fm_0^j\},
\end{equation*}
This is a valuation if $0$ is a smooth point, but not in general.
See~\S\ref{S113} for more details.

Recall that a valuation $v$ of the function field $k(Y)$ is \emph{divisorial}
if there exists a projective birational morphism $X\to Y$, with $X$ normal,
and an irreducible Cartier divisor $E$ on $X$ such that $v$ is proportional 
to $\ord_E$, the order of vanishing along $E$. We say that $v$ is 
\emph{centered at 0} if $E$ lies above 0, or, equivalently, 
$v\ge0$ on $\cO_{Y,0}$ and $v>0$ on $\fm_0$.

Izumi's Theorem\footnote{In fact the original statement is slightly more general
assuming only $(Y,0)$ to be analytically irreducible.}
says that any divisorial valuation centered at 0
is comparable to the order of vanishing at 0:
\begin{IzumisThm}
  For any divisorial valuation $v$ of $k(Y)$ centered at 0 
  there exists a constant $C=C(v)>0$ such that 
  \begin{equation}
    C^{-1}\ord_0(f)\le v(f) \le C\ord_0(f).\tag{*}
  \end{equation}
\end{IzumisThm}
Only the right-hand inequality in~(*) is nontrivial. 
Indeed, if we set 
$c=v(\fm_0)=\min\{v(f)\mid f\in\fm_0\}$ then 
$c>0$ and $v\ge c\ord_0$.

Several versions of Izumi's Theorem can be found in the literature.
In the case when $k$ is of characteristic zero and 
$Y$ is smooth, it goes back at least to Tougeron, 
see~\cite[p.178]{tougeron} (the same proof  was used in the
context of plurisubharmonic functions by the second author in~\cite{lelong}).
A proof based on multiplier ideals is given in~\cite{ELS}. These approaches give an estimate
on the optimal constant $C$ in~(*)  in terms of log-discrepancies.

Izumi himself was mainly interested in the case of singular complex
analytic spaces, see~\cite{izumi1,izumi}. 
His argument has been generalized by Rees~\cite{Ree89}, and alternative proofs given by 
H\"ubl and Swanson~\cite{HuSw01}, and Beddani~\cite{beddani}. 
Another approach, based on the notion of key polynomials,
was recently developed by  Moghaddam~\cite{Moghaddam}, see~\cite{valtree} in the two-dimensional case. 
For a connection between Izumi's Theorem and the Artin-Rees Lemma,
see~\cite{Ron06}.

\bigskip
Our objective is not to generalize the setting of Izumi's Theorem, but to make the
statement more precise. 
Consider a projective birational morphism $\pi:X\to Y$ with $X$ 
smooth.
We assume that $Z:=\pi^{-1}(0)$ is a divisor
with simple normal crossing support such that any nonempty 
intersection between irreducible components of $Z$ is irreducible.
(We do not assume that $Z$ is reduced and 
the exceptional set of $\pi$ may be strictly larger than $Z$.
If $k$ has characteristic zero, the existence of such a morphism
follows from Hironaka's Theorem.)

The \emph{dual complex} $\Delta=\Delta(X,Z)$ is a simplicial complex encoding the 
intersections of the irreducible components of $Z$.
We can view the elements of $\Delta$ as 
\emph{quasimonomial valuations} on $\cO_{Y,0}$ 
centered at 0, see~\eg\cite{jonmus}.\footnote{Quasimonomial valuations are also known as Abhyankar valuations.}
There is a natural (integral) affine structure on $\Delta$.
Pick a metric on $\Delta$ that is compatible with this structure.

Any function $f\in\cO_{Y,0}$ defines a nonnegative 
function on $\Delta$ given by $v\mapsto v(f)$.
\begin{thmA}
  There exists a constant $A>0$ such that for any $f\in\cO_{Y,0}$, 
  the function $v\mapsto v(f)$ on $\Delta$ is concave on each face and 
  Lipschitz continuous with Lipschitz constant at most $A\ord_0(f)$.
\end{thmA}
The constant $A$ depends on $X$ and on the metric on $\Delta=\Delta(X,Z)$
but not on $f$. 

It is not hard to see that Theorem~A implies Izumi's Theorem
in the case when the base field $k$ has characteristic zero.
Indeed, in view of Hironaka's theorem~\cite{Hir64}, 
any divisorial valuation on $k(Y)$ centered at 0
is proportional to a point $v$ in some dual complex $\Delta$
of some $X$ as above; we can even choose $v$ as a vertex.
Further, $Z$, and hence $\D$, is connected as a consequence of 
Zariski's Main Theorem. 
By Theorem~A we have 
\begin{equation*}
  \max_{v\in\D}v(f)\le (1+A\diam(\D))\min_{v\in\D}v(f).
\end{equation*}
On the other hand, one can show (see~\S\ref{SpfIzu}) 
that $\min_{v\in\D}v(f)$ is comparable to 
$\ord_0(f)$; hence Izumi's Theorem follows.

\medskip
One can also rephrase Theorem~A in terms of Newton polyhedra.
See~\S\ref{S116} for details on what follows.
Let $E_i$, $i\in I$ be the irreducible components of $Z$.
Pick $J\subset I$ such that 
$E_J:=\bigcap_{j\in J}E_j\ne\emptyset$, and let $z_j\in\cO_{X,\xi_J}$, 
$j\in J$ be a system of coordinates at the generic point $\xi_J$ of $E_J$
such that $E_j=\{z_j=0\}$ for $j\in J$. 
Using Cohen's Theorem, we can expand
any $f\in\cO_{Y,0}\subset\cO_{X,\xi_J}$ as a formal power series
in the $z_j$ with coefficients in the residue field of $\xi_J$.
Let $\Nw(f,J)\subset\R_{\ge0}^J$ be the Newton polyhedron of this expansion.
Fix a norm on $\R^J$.
\begin{thmAp}
  There exists a constant $A>0$ such that for any $f\in\cO_{Y,0}$
  and any $J$ with $E_J\ne\emptyset$, 
  all extremal points of the Newton polyhedron $\Nw(f,J)$
  are of norm at most $A\ord_0(f)$.
\end{thmAp}

Theorem~A will be a consequence of a more general result that we now describe.
Let $X$ be a smooth, quasiprojective variety over $k$ and 
$Z\subset X$ an effective divisor with proper  and connected simple normal crossing 
support such that any nonempty intersection
between irreducible components of $Z$ is irreducible.
We view the elements of the dual complex 
$\Delta=\Delta(X,Z)$ as rank $1$ valuations
on the function field of $X$ normalized by $v(Z)=1$. 

Fix an effective  
divisor $G$ on some open neighborhood of
$Z$ in $X$. 
We can define $v(G)$ for any $v\in\D$ using local defining equations
of $G$.
Thus $G$ gives rise to a function $\chi=\chi_G$ on $\D$,
defined by $\chi(v):=v(G)$.
Fix a line bundle $M\in\Pic(X)$ that is ample on $Z$.
\begin{thmB}
  There exist constants $A$ and $B$ such that for any $G$ as above,
  the function $\chi=\chi_G$ on $\Delta$ is concave on each face and 
  Lipschitz continuous with Lipschitz constant at most 
  \begin{equation*}\label{eqB}
    A\min_\Delta\chi+B\max_J|(G\cdot M^{n-|J|-1}\cdot E_J)|, \tag{**}
  \end{equation*}
  where the maximum is over subsets $J\subset I$ for which $E_J\ne\emptyset$.
  Here the constants $A$ and $B$ depend on $X$, $M$ and the metric chosen on 
  $\Delta=\Delta(X,Z)$, but not on $G$.
\end{thmB}

Theorem~A follows from Theorem~B by picking $G$ as the divisor of 
$f\circ\pi$.
Indeed, the second item in~\eqref{eqB} vanishes, and one can show that $\min_\D\chi$ and $\ord_0(f)$ are comparable.

\medskip
Theorem~B can also be applied to study polynomials at infinity.
Fix an embedding $\A^m\subset\P^m$. 
Following the terminology introduced in~\cite{dyncomp,dynberko} in a dynamical context, 
we say that an \emph{admissible compactification} of $\A^m$
is a smooth projective variety $X$ over $k$ 
together with a projective birational morphism $\pi:X\to\P^m$
that is an isomorphism over $\A^m$ and such that if  $Z$ is
the pullback of  the hyperplane at infinity $\P^m\setminus\A^m$, then $Z$ has simple normal crossing support
and any nonempty intersection between irreducible components of
$Z$ is irreducible. By Zariski's Main Theorem, the support of $Z$ is 
connected.
We can view the elements of the dual complex $\Delta=\Delta(X,Z)$ as
valuations on $k(X)=k(\P^m)$ as above. 
In particular, any polynomial $P\in k[\A^m]$ defines a function 
on $\Delta$, given by $v\mapsto v(P)$. This function is
easily seen to be concave and piecewise affine on the faces of $\D$.
\begin{corC}
  There exists a constant $B>0$ such that if $P\in k[\A^m]$ is a
  polynomial of degree $d\ge1$, then the function
  $v\mapsto v(P)$ on $\D$ is
  Lipschitz continuous with Lipschitz constant 
  at most $Bd$.
\end{corC}
This follows by taking the divisor $G$ as the pullback to $X$ of the hypersurface 
on $\P^m$ defined as the zero locus of $P$.
We have $v(P)=v(G)-d=\chi_G(v)-d$ 
for $v\in\D$, so the Lipschitz constant of $v\mapsto v(P)$
is the same as that of $\chi_G$. 
Now, $\min_\D\chi_G=0$, so Corollary~C follows from Theorem~B.

\medskip

Finally, we use Theorem~A in order to study the variation of several natural numerical 
invariants associated to rank $1$ valuations. As above, let $Y$ be a normal variety of dimension $m$, defined over an algebraically closed field $k$, and
let $0$ be a closed point of $Y$.
Pick any two $\fm_0$-primary ideals $\fa_1, \fa_2 \subset \cO_{Y,0}$, and denote by $e(\fa_i)$ their Hilbert-Samuel multiplicities.
It is a theorem due to Teissier and Risler~\cite[\S 2]{cycles evanescents} that 
the function $(r,s)\mapsto e(\fa_1^r\cdot \fa_2^s)$ is a homogeneous polynomial of degree $m$
and that we can find nonnegative integers $e(\fa_1^{[m-i]};\fa_2^{[i]})$, $0\le i\le m$, such that 
$$
e(\fa_1^r\cdot \fa_2^s) = \sum_{i=0}^m {m\choose i} \, e(\fa_1^{[m-i]};\fa_2^{[i]})\, r^{m-i}s^i
$$
for all $r,s \in \Z_{\ge0}$.

Pick any  rank $1$ valuation $v$ on $\cO_{Y,0}$ centered at $0$. Then  the sequence of valuation ideals $\fa(v,n)=\{f\in\cO_{Y,0}\mid v(f)\ge n\}$
forms a graded sequence in the sense that 
$\fa(v,n)\cdot \fa (v,n')\subset \fa (v,n+n')$ for any $n,n'$.
One can show (see \S\ref{S130} below) that for any integer $0\le i\le m$, the following limit exists:
\begin{equation*}
  \a_i(v):=\lim_{n\to\infty} \frac{e(\fa (v,n)^{[i]}; \fm_0^{[m-i]})}{n^i}~.
\end{equation*}

When $i=m$, it is 
a theorem due to~\cite{ELS,LazMus} (see also~\cite{cutkosky-mult}) show that 
the sequence $\frac{m!}{n^m}\, \dim_k(\cO_{Y,0}/\fa(v,n))$ converges and that its
limit is equal to $\a_m(v)$. This invariant is usually referred to as the 
\emph{volume} of a valuation.
This invariant is quite subtle, since it can be irrational
even when the valuation is divisorial, see~\cite[example 6]{cutkosky-srinivas} or~\cite{Kur03}.

\smallskip

Now let $X$, $Z$ and $\D=\D(X,Z)$ be as in Theorem~A.
\begin{corD}
For any $0\le i\le m$,  the function $v\mapsto\a_i(v)$ is 
  Lipschitz continuous on $\Delta$.
\end{corD}
This result is new even in the case $i=m$. Note that Fulger~\cite{fulger} has introduced a notion of local volume for  divisors on $X$, and proved that this local volume
is locally Lipschitz in the relative N\'eron-Severi space $N(X/Y)$, see Proposition 1.18 in ibid. It is unlikely that one can recover the Lipschitz continuity of $\a_m$ through his result since there is no canonical way to attach to a valuation $v\in\D(X,Z)$ a divisor in $X$ 
that  computes $\a_m(v)$.
\footnote{When $0$ is an isolated singularity, $Z(v):= \lim\frac1n Z(\fa(v,n))$ is  a nef $b$-divisor over $0$ in the sense of ~\cite{BdFF}, and one can show that $\a_m(v) = - Z(v)^m$. However it is unclear how to use this interpretation to prove the continuity of $\a_m$.}
His result is, however, close in spirit to the continuity statement
for the (global) volume function on the N\'eron-Severi space of a projective variety, see~\cite{LazBook}.
The latter statement has been strengthened in~\cite{diffvol} to show that the global volume function is in fact differentiable on the N\'eron-Severi space. By analogy, one can  ask  whether or not Fulger's local volume, and the functions $v\mapsto \a_i(v)$
are differentiable.

\smallskip

In the case $Y$ is smooth at $0$,  we have
$$\a_1(v) = \left(\sup_{\fm_0} \frac{v}{\ord_0}\right)^{-1}~.$$ 
In the general singular case, 
one can find a constant $C>0$ such that  
$$C^{-1}\, \a_1(v) \le \left(\sup_{\fm_0} \frac{v}{\ord_0} \right)^{-1}\le C \a_1(v)$$ for all $v$,
see Proposition \ref{p8}. In particular, Corollary~D gives a control on the variation of the optimal constant appearing in Izumi's theorem.
More generally we obtain
\begin{corE}
The function $(v,v')\mapsto \sup_{\fm_0} \frac{v}{v'}$ 
is  Lipschitz continuous on $\Delta\times \D$.
\end{corE}
The constant $\sup_{\fm_0} \frac{v}{v'}$ is sometimes referred to as the linking number 
of  two valuations, see~\cite{huckaba,samuel}.

In dimension $2$ over a smooth point, then $\a_2(v) = \vol (v)$
is equal to $(\sup_{\fm_0} v/ \ord_0)^{-1}$ by~\cite[Remark 3.33]{valtree} for any
 valuation normalized by $v(\fm_0) = +1$. In this case, we thus have
$$
\a_2(v) = \vol (v), \,\, \a_1(v) = \left(\sup_{\fm_0} \frac{v}{\ord_0}\right)^{-1},
\text{ and } \a_2(v) = \frac{\a_1(v)}{v(\fm_0)}~.
$$
Observe that since the function $\sup_{\fm_0} \frac{v}{\ord_0}$ is affine on (each segment of) the dual graph 
$\D$ by~\cite[\S 6]{valtree}, it follows that $\a_1$ and $\a_2$ are both differentiable functions on $\D$.

\bigskip
Our approach to Theorem~B follows~\cite{siminag}, where a similar
result was proved in a slightly different context.
The fact that $\chi$ is continuous, concave and piecewise affine
on the faces on $\D=\D(X,Z)$ is a direct consequence of the
way $\D$ is embedded into the set of valuations on
the function field on $X$.
After this observation the proof consists of two steps. 

First we give an upper bound for $\chi$ on the vertices of $\D$. 
Our argument for this uses elementary intersection theory and 
in fact is quite close to the original proof of Izumi's Theorem
by Izumi himself.

Second, we prove the Lipschitz estimate. Because of the concavity,
it suffices to bound certain directional derivatives of $\chi$ from
above. To do this, we first define a suitable simplicial subdivision
$\D'$ of $\D$ such that $\chi$ is affine on (a suitable subset of)
the faces of $\D'$. 
Using the toroidal techniques of~\cite{KKMS}, we can associate 
to $\D'$ a projective birational morphism $X'\to X$, where $X'$ 
is a normal, $\Q$-factorial variety. Roughly speaking, the directional
derivatives on $\D$ translate into actual differences on $\D'$,
and these can be estimated more or less as in the first step.

\medskip
One of our motivations behind this paper is to study pluripotential 
theory on Berkovich spaces~\cite{BerkBook} over a field equipped with a trivial norm.
The Lipschitz estimate in Theorem~B implies the compactness of
certain spaces of quasi-plurisubharmonic functions that appear
in~\cite{hiro,BdFF}.
These applications to pluripotential theory will appear
elsewhere; the corresponding results (including the Lipschitz estimate)
for a discretely valued field can be found in~\cite{siminag,nama}.
For more on (pluri)potential theory in a non-Archimedean setting,
see also~\cite{BRBook,thuphd,FR,valtree,dynberko}.

\bigskip
The paper is organized as follows.
In~\S\ref{S120} we recall some basic facts about valuations
in general and quasimonomial valuations in particular.
We also state a result that follows from~\cite{KKMS}.
In~\S\ref{S101} we recall some facts about Lipschitz constants for
convex functions. The proof of our main result, Theorem~B,
is then given in~\S\ref{S114} whereas its various consequences are 
established in~\S\ref{S121}. 
%
%
\section{Background}\label{S120}
Throughout the paper, $k$ is an algebraically closed field.
By a variety over $k$ we mean a separated integral scheme of finite
type over $k$.
If $Z$ is a subscheme of a scheme $X$, we denote by $|Z|$ its support.
%
%
\subsection{Valuations}\label{S1118}
Let $X$ be a normal, quasiprojective variety over $k$.
By a \emph{valuation on $X$} we mean a (rank 1) valuation 
$v:k(X)\to\R$ that is trivial on $k$ and admits a \emph{center} on $X$, 
that is, a point (not necessarily closed) $\xi\in X$ such that 
$v$ is nonnegative on the local ring $\cO_{X,\xi}$
and strictly positive on the maximal ideal of this ring.
Since $X$ is assumed separated, the center is unique if it exists.
We write $\Val_X$ for the set of all valuations on $X$.
For a closed point $0\in X$, we shall also denote by $\Val_{X,0}$
the subset of valuations $v\in \Val_X$ such that $v(\fm_0) >0$.

If $G$ is a $\Q$-Cartier divisor on $X$ and $v\in\Val_X$, then we
define $v(G):=\frac1mv(f_m)$, where $m\in\Z_{>0}$ is such that $mG$ is
a Cartier divisor and $f_m\in\cO_{X,\xi}$ is a local equation for $mG$
at the center $\xi$ of $v$ on $X$. If $G$ is effective, then $v(G)\ge0$
with strict inequality if and only if $\xi$ is contained in the support of $G$.
 
Consider a proper birational map $\pi:X'\to X$ with $X'$ normal.
If $E\subset X'$ is a prime divisor, then $\ord_E$, the order
of vanishing along $E$ defines an element of $\Val_{X'}=\Val_X$. 
Any valuation proportional to such a valuation will be called \emph{divisorial}.

%
%
\subsection{Dual complexes}\label{S104}
Now assume $(X,Z)$ is an \emph{SNC pair}. 
By this we will mean that $X$ is a smooth, quasiprojective 
variety over $k$ and $Z\subset X$ is an effective divisor with projective, connected,  
simple normal crossing support such that any nonempty intersection of
irreducible components of $Z$ is connected.
Thus we can write $Z=\sum_{i\in I}b_iE_i$,
where $E_i$, $i\in I$ are the irreducible components of $|Z|$,
$b_i\in\Z_{>0}$
and, for any $J\subset I$,
the intersection $E_J:=\bigcap_{j\in J}E_j$ is either empty or irreducible.

The \emph{dual complex} $\D=\D(X,Z)$ is a simplicial complex defined in the 
usual way: to each $i\in I$ is associated a vertex $e_i$ and to each 
$J\subset I$ with $E_J\ne\emptyset$ is associated a simplex $\sigma_J$
containing all the $e_j$, $j\in J$.

Let $\Div(X,Z)\simeq\bigoplus_{i\in I}\Z E_i$ 
be the free abelian group of divisors on $X$ supported on $|Z|$.
Set $\Div(X,Z)_\R:=\Div(X,Z)\otimes_\Z\R\simeq\bigoplus_{i\in I}\R E_i$.
We can embed $\D$ in the dual vector space $\Div(X,Z)_\R^*$ as follows.
A vertex $e_i$ of $\D$ is identified with the element in $\Div(X,Z)_\R^*$ satisfying 
$\langle e_i,E_i\rangle=b_i^{-1}$ and
$\langle e_i,E_j\rangle=0$ for $i\ne j$.
A simplex $\sigma_J$ of $\D$ is identified with the 
convex hull of $(e_j)_{j\in J}$ in $\Div(X,Z)_\R^*$.
In this way, $\D$ can be written
\begin{equation*}
  \D=\left\{
    t=\sum_{i\in I}t_ie_i\
    \bigg|\ t_i\ge0, \sum_ib_it_t=1, \bigcap_{i\mid t_i>0}E_i\ne\emptyset
  \right\}
  \subset\Div(X,Z)^*_\R.
\end{equation*}
This embedding naturally equips $\D$ with an 
\emph{integral affine structure}: the integral affine functions are
the restrictions to $\D$ of the elements in $\Div(X,Z)$.
%
%
\subsection{Quasimonomial valuations}\label{S119}
We can also embed the dual complex $\D$ into the 
valuation space $\Val_X$. See~\cite[\S3]{jonmus}
for details on what follows.

Pick a point $t=\sum_{i\in I}t_ie_i\in\D\subset\Div(X,Z)^*_\R$, 
let $J$ be the set of indices $j\in I$ such that $t_j>0$
and let $\xi_J$ be the generic point of $E_J=\bigcap_{j\in J}E_j$.
Pick local algebraic coordinates $z_j\in\cO_{X,\xi_J}$, $j\in J$, 
such that $E_j=\{z_j=0\}$. We then associate to $t$ the valuation $\val_t$,
which is a monomial valuation in these coordinates with weight $t_j$ on $z_j$, $j\in J$.
More precisely, $\val_t$ is defined as follows.
Using Cohen's Theorem, we can write any $f\in\cO_{X,\xi_J}$
in the complete ring $\widehat{\cO_{X,\xi_J}}$ as a formal power series 
\begin{equation}\label{e107}
  f=\sum_{\a\in\N^J}f_\a z^\a.
\end{equation}
where $f_\a\in\widehat{\cO_{X,\xi_J}}$
and, for each $\a$, either $f_\a=0$ or $f_\a(\xi_J)\ne0$.
We then set
\begin{equation}\label{equ:monoval}
  \val_t(f):=\min\{\langle t,\alpha\rangle\mid f_\a\ne0\}.
\end{equation}
While the expansion~\eqref{e107} is
not unique, one can show that~\eqref{equ:monoval} is well defined.
Further, it suffices to take the minimum over finitely many $\a$.
If $t\in\D$, then the center of $\val_t$ on $X$ is the generic
point of $E_J$, where $J\subset I$ is defined by the property that 
$v$ lies in the relative interior of $\sigma_J$.
\begin{prop}\label{P102} 
  Let $(X,Z)$ be an SNC pair. Then, for any effective divisor $G$ on
  $X$, the function $v\mapsto v(G)$ is continuous, concave and integral piecewise
  affine on $\D=\D(X,Z)$.
\end{prop}
\begin{proof} 
  The function $v\mapsto v(G)$ is continuous on $\Val_X$, so its
  restriction to $\D$ is also continuous.
  Let $\sigma=\sigma_J$ be a face of $\D$, determined by a subset
  $J\subset I$ such that $E_J\ne\emptyset$.
  Let $\xi=\xi_J$ be the generic point of $E_J$
  and $f\in\cO_{X,\xi}$ a defining equation for $G$ at $\xi$.
  It then follows from~\eqref{equ:monoval} that $t\mapsto\val_t(f)$ 
  is continuous, piecewise integral affine and convex on $\sigma_J$.
 \end{proof}
The valuation $\val_t$ is divisorial if and only if
$t_j\in\Q$ for all $j$, see~\cite[Remark~3.9]{jonmus}.
In particular, the set of $t\in\D$ for which $\val_t$ is divisorial
is dense in $\D$.
%
%
\subsection{Subdivisions and blowups}\label{S102}
A \emph{subdivision} $\D'$ of $\D=\D(X,Z)$ is a compact rational polyhedral 
complex of $\Div(X,Z)_\R^*$ refining $\D$. 
A subdivision $\D'$ is \emph{simplicial} if its faces are simplices.
It is \emph{projective} if there exists a convex, piecewise integral 
function $h$ on $\D$ such that 
$\D'$ is the coarsest subdivision of $\D$ 
on each of whose faces $h$ is affine.
(Such a function $h$ is called a support function for $\D'$.)
\begin{thm}\label{T102}
  Let  $\D'$ be a simplicial projective subdivision of $\D$. 
  Then there exists a projective birational morphism 
  $\rho:X'\to X$ with the following properties:
  \begin{itemize}
  \item[(i)]
    $\rho$ is an isomorphism on $X'\setminus|Z'|$,
    where $Z':=\rho^{-1}(Z)$;
  \item[(ii)] 
    $X'$ is normal, $Z'$ has pure codimension 1 and 
    every irreducible component of $Z'$ is $\Q$-Cartier;
  \item[(iii)] 
    the vertices $(e'_i)_{i\in I'}$ of $\D'$ are in bijection with the irreducible components 
    $(E'_i)_{i\in I'}$ of $Z'$: for each $i\in I'$, the center on $X'$ of $e'_i$ 
    is the generic point of $E'_i$;
  \item[(iv)]
    If $J'\subset I'$, then $E'_{J'}:=\bigcap_{j\in J'}E'_j$ is 
    nonempty if and only if the corresponding vertices  $e'_j$, $j\in J'$ of $\Delta'$ 
    span a face $\sigma'_{J'}$ of $\Delta'$;
    in this case, $E'_{J'}$ is normal, irreducible, of codimension $|J'|$,  
    and its generic point is the 
    center of $v$ on $X'$ for all $v$ in the relative interior of $\sigma'_{J'}$;
  \item[(v)] 
    for each $i\in I'$, the function $\D\ni v\to v(E'_i)\in\R$
   is affine on the simplices of $\D'$.
 \end{itemize}
\end{thm}
Since $Z$ is a divisor with simple normal crossing singularities on an smooth ambient space, 
the inclusion $X \setminus Z \subset X$ is a toroidal embedding in the sense of~\cite[Chapter~II]{KKMS}.
This result is thus a consequence of the toroidal analysis in op. cit.
%
%
\section{Some convex analysis}\label{S101}
In this section we note some basic facts about convex functions.
Let $V$ be a finite dimensional real vector space and $\tau\subset V$ a 
compact convex set containing at least two points. 
Denote by $\cE(\tau)$ the set of 
extremal points of $\tau$. 

Given a norm $\|\cdot\|$ on $V$, the Lipschitz constant 
of a continuous function $\f:\tau\to\R$ is defined as usual as
\begin{equation*}
\lip_\tau(\f):=\sup_{v\neq v'}\frac{|\f(v)-\f(v')|}{\|v-v'\|}\in [0,+\infty]
\end{equation*}
and its $C^{0,1}$-norm is then 
\begin{equation*}
\|\f\|_{C^{0,1}(\tau)}:=\|\f\|_{C^0(\tau)}+\lip_\tau(\f) ,
\end{equation*}
where $\|\f\|_{C^0(\tau)} := \sup_\tau |\f|$.
This quantity of course depends on the choice of $\|\cdot\|$, but since all norms on $V$ are equivalent, choosing another norm only affects the estimates to follow by an overall multiplicative constant. 
%
%
\subsection{Directional derivatives}
Now let $\f:\tau\to\R$ be convex and continuous.
For $v,w\in\tau$ we define the
directional derivative of $\f$ at $v$ towards $w$ as
\begin{equation}\label{equ:der}
  D_v\f(w):=\left.\frac{d}{dt}\right|_{t=0_+}\f((1-t)v+tw);
\end{equation}
this limit exists by convexity of $\f$. 
\begin{lem}\label{lemusc}
For any fixed $w\in \tau$, the function $v \mapsto D_v\f (w)$ is upper semicontinuous. 
\end{lem}
\begin{proof}
Fix $v$ and let $\e>0$.
Then there exists $0<t<1$ such that  
$$D_v \f(w) \ge \frac{\f(tw + (1-t)v) - \f(v)}{t} - \e~.$$
Since $\f$ is continuous, we have
$$\frac{\f((1-t)v + tw ) - \f(v)}{t}  \ge \frac{\f((1-t)v' + tw) - \f(v')}{t}  - \e$$
for any $v'$ close to $v$.
Now $\f$ is convex hence
$$\frac{\f((1-t)v' + tw ) - \f(v')}{t}  \ge D_{v'}\f(w),$$ and we conclude that 
$$
D_v\f(w) \ge D_{v'}\f(w) - 2\e
$$
for any $v'$ close to $v$.
This ends the proof.
\end{proof}

\begin{prop}\label{prop:lip}
  There exists $C>0$ such that every Lipschitz continuous 
  convex function $\f:\tau\to\R$ satisfies
  \begin{equation*}
    C^{-1}\|\f\|_{C^{0,1}(\tau)}
    \le \|\f\|_{C^0(\partial\tau)}
    +\sup_{e\in\cE(\tau),v\in\mathrm{int}(\tau)}\left|D_{\pi_e(v)}\f(e)\right|
    \le C\,\|\f\|_{C^{0,1}(\tau)}.
  \end{equation*}
  Here $\pi_e(v)\in\partial\tau$ is the 
  unique point in $\partial\tau$ such that $v\in[e,\pi_e(v)]$. 
\end{prop}
Observe that $\sup\{ \left|D_{\pi_e(v)}\f(e)\right|, \,
e\in\cE(\tau),v\in\mathrm{int}(\tau)\}$ equals  $\sup \{ \left|D_w\f(e)\right|,\, 
e\in\cE(\tau),\, w\in\partial \tau,\, [w,e]\not\subset \partial \tau\}$.

For the proof, see~\cite[Lemma~A.2]{siminag}.
%
%
\subsection{Newton polyhedra}\label{S115}
Assume now that $\tau\subset V$ is a compact polytope whose affine span 
$\langle\tau\rangle$ is an affine hyperplane 
that does not contain the origin of $V$. 
Let $\f:\tau\to\R$ be a piecewise affine continuous convex function. 
It extends as a $1$-homogeneous piecewise linear convex function 
on the polyhedral cone $\hat\tau$ over $\tau$, 
whose \emph{Newton polyhedron} $\Nw(\f)$ is as usual 
defined as the convex subset of $V^*$ consisting 
of all linear forms $m\in V^*$ such that 
$m\le\f$ on $\hat\tau$ (or, equivalently, on $\tau$).
We endow $V^*$ with the dual norm
$\|m\| := \sup_{\|v \| =1} \langle m,v \rangle$.

\begin{prop}\label{P114}
  There exists a constant $C>0$, not depending on $\f$, such that
  \begin{equation*}
    C^{-1}\,\|\f\|_{C^{0,1}(\tau)}\le\max_{m\in\cE_\tau(\f)}\|m\|\le C\,\|\f\|_{C^{0,1}(\tau)},
  \end{equation*}
  where $\cE_\tau(\f)\subset V^*$ denotes the (finite) set of extremal points 
  of the Newton polyhedron of $\f$.
\end{prop}
\begin{proof} 
  Since $\tau$ is a non-empty compact subset disjoint from the linear hyperplane 
  $W$ parallel to $\langle\tau\rangle$, it is clear by homogeneity 
  that there exists $C>0$ such that 
  \begin{equation}\label{equ:normes}
    \left\|m\right\|\le C\left(\left\|m|_W\right\|
      +\inf_{v\in\tau}|\langle m,v\rangle|\right)
  \end{equation}
  for all $m\in M_\R$. On the other hand, elementary convex analysis tells us that
  \begin{equation}\label{equ:fmax}
    \f(v)=\max_{m\in\cE_\tau(\f)}\langle m,v\rangle
  \end{equation}
  for all $v\in\tau$, and that the set $\{v\in\tau\mid\f(v)=\langle m,v\rangle\}$ 
  has non-empty interior in $\tau$ for each $m\in\cE_\tau(\f)$. 
  We thus see that the image of the gradient of $\f$ on its differentiability locus 
  is exactly the finite set
  $\{m|_W\mid m\in\cE_\tau(\f)\}\subset W^*$, which implies that the 
  Lipschitz constant of $\f$ on $\tau$ satisfies
  \begin{equation*}
  \lip_\tau(\f)=\max_{m\in\cE_\tau(\f)}\left\|m|_W\right\|. 
  \end{equation*}
Since $\| \f \|_{C^{0}(\tau)} \le \max_\tau \|v\| \,  \max_{\cE_\tau(\f)} \| m \|$, and $\| m|_W\| \le \|m \|$,
the left-hand inequality is now clear. Since for each $m\in\cE_\tau(\f)$ 
  there exists $v\in\tau$ such that $\f(v)=\langle m,v\rangle$, we have 
  $\inf_{v\in\tau}|\langle m,v\rangle|\le\|\f\|_{C^0(\tau)}$, 
  and~\eqref{equ:normes} yields the right-hand inequality.
\end{proof}
%
%
\section{Proof of Theorem B}\label{S114}
Write $Z=\sum_{i\in I}b_iE_i$.
Fix a line bundle $M$ on $X$ that is ample on $Z$ and set
\begin{equation*}
  \theta_G:=\max_{J\subset I}|(G\cdot E_J\cdot M^{n-|J|-1})|.
\end{equation*}
Note that $\theta_G=0$ if the line bundle $\cO_X(G)|_Z$ is trivial.

Throughout the proof, $A\ge1$ and $B\ge0$ will denote various 
constants whose values may vary from line to line,
but they do not depend on $G$.

We already know from Proposition~\ref{P102} that the function
$\chi=\chi_G$ is nonnegative, concave and integral 
 piecewise affine on each simplex in $\D$.
%
%
\subsection{Bounding the values on vertices} 
We first prove the estimate
\begin{equation}\label{e102}
  \chi(e_i)\le A\min_\D\chi+B\theta_G
  \quad\text{for all $i\in I$}.
\end{equation}
Since $\chi$ is concave on each simplex,
its minimum on $\D$ must be attained at a vertex. 
Further, the $1$-skeleton of $\D$ is connected since $Z$ has connected support. 
We may therefore assume that $I=\{0,1,\dots,m\}$, where
$\chi(e_0)=\min\chi$, 
$e_i$ is adjacent to $e_{i+1}$ (\ie $E_i\cap E_{i+1}\ne\emptyset$) 
for $i=0,\dots, l-1$
and $\chi(e_l)=\max_{i\in I}\chi$, where $1\le l\le m$.
It suffices to prove~\eqref{e102} for $1\le i\le l$ and
this we shall do by induction.
Let us write
\begin{equation*}
  G=\sum_{j\in I}b_j\chi(e_j)E_j+\tilde{G},
\end{equation*}
where $\tilde{G}$ is an effective divisor whose support 
does not contain any $E_i$.
For each $i\in I$ we then have
\begin{align}
  \sum_{j\in I}b_j\chi(e_j)(E_i\cdot E_j\cdot M^{n-2})
  &=(G\cdot E_i\cdot M^{n-2}) 
  -(\tilde{G}\cdot E_i\cdot M^{n-2})\notag\\
  &\le (G\cdot E_i\cdot M^{n-2})\notag\\
  &\le\theta_G.\label{e101} 
\end{align}
Set 
\begin{equation*}
  c_{ij}:=b_j(E_i\cdot E_j\cdot M^{n-2})
\end{equation*}
for $i,j\in I$.
Note that $c_{ij}\ge 0$ for $j\neq i$,
with strict inequality if and only if $e_i$ and $e_j$ are adjacent.
In particular, $c_{i,i+1}>0$ for $0\le i<l$.
Since $\chi\ge0$ it follows from~\eqref{e101} that 
\begin{align*}
  \chi(e_{i+1})
  &\le 
  \frac{|c_{ii}|}{c_{i,i+1}}\chi(e_i)
 +\frac{\theta_G}{c_{i,i+1}}\\
  &\le A_0\chi(e_i)+B_0\theta_G,
\end{align*}
for $0\le i<l$, where the constants $A_0$ and $B_0$ do not depend on
$i$ or $G$.
We may assume that $A_0\ge 1$.
A simple induction now gives~\eqref{e102}
for $0\le i< l$, with $A=A_0^l$ and $B=B_0(1+A_0+\dots+A_0^{l-1})$.
%
%
\subsection{Bounding Lipschitz constants} 
Let $\tau$ be a face of $\D$. 
Our aim is to prove by induction on $\dim\tau$ that
\begin{equation*}
  \|\chi|_\tau\|_{C^{0,1}}
  \le A\min_\D\chi+B\theta_G.
\end{equation*}
Here the $C^{0,1}$-norm is defined as the sum of 
the sup-norm and the Lipschitz constant; see~\S\ref{S101}.

The case $\dim\tau=0$ is settled 
by~\eqref{e102} so let us assume that $\dim\tau>0$.
By Proposition~\ref{P102} the restriction of $\chi$ to $\tau$ 
is piecewise affine and concave. 
It therefore admits directional derivatives, and we set as in~\eqref{equ:der}
\begin{equation*}
  D_v\chi(w):=\left.\frac{d}{dt}\right|_{t=0+}\chi\left((1-t)v+tw\right)
\end{equation*}
for $v,w\in\tau$. 

Let us say that a codimension $1$ face of $\tau$ is \emph{opposite} a vertex when it is the convex hull of the remaining vertices of $\tau$. This notion is well-defined since $\tau$ is a simplex. 
\begin{prop}\label{prop:keyprop}
  We have 
  \begin{equation*}
    D_v\chi(e) \le A\min_\D\chi+B\theta_G
  \end{equation*}
  for any vertex $e$ of $\tau$, and any rational point 
  $v$ in the relative interior of the face $\sigma$ of $\tau$ 
  opposite to $e$, such that $\chi|_\sigma$ is affine near $v$.
\end{prop}
Granting this result, let us explain how to conclude the proof of Theorem~C.
By induction we have $\sup_{\partial\tau}\chi\le A\min_\D\chi+B\theta_G$.
The fact that $\chi$ is concave and nonnegative 
implies that
\begin{equation*}
  D_v\chi(e)\ge\chi(e)-\chi(v)\ge-\min\chi
\end{equation*}
for any  $e,v \in \partial \tau$.
By Proposition~\ref{prop:keyprop} this gives 
(assuming, as we may, that $A\ge1$)
\begin{equation*}\label{eqbd}
  |D_v\chi(e)| \le A|\min_\D\chi|+B\theta_G \tag{\dag}
\end{equation*}
for any vertex $e$ of $\tau$ and any rational point $v$ 
in the relative interior of the face $\sigma$ opposite to $e$ such that 
$\chi|_\sigma$ is affine near $v$. 

By Lemma~\ref{lemusc} applied to the convex function $-\chi$, 
the function $v \mapsto D_v\chi(e)$ is lower semicontinuous on $\sigma$.
It follows  by density that the upper bound~\eqref{eqbd} holds 
for any $v$ in the relative interior of $\sigma$. 
We conclude by Proposition~\ref{prop:lip}
that the $C^{0,1}$-norm of $\chi|_\tau$ is bounded by $A\min_\D\chi+B\theta_G$,
completing the proof of Theorem~C.

The rest of~\S\ref{S114} is devoted to the proof of Proposition~\ref{prop:keyprop}.
%
%
\subsection{Special subdivisions}\label{S103}
The \emph{star} of a face $\sigma$ of $\D$ is defined as usual as the 
subcomplex $\sta(\sigma)$ of $\D$ made up of all the faces of $\D$ containing $\sigma$. 
A vertex $e_i$ thus belongs to the star of a simplex $\sigma_J$ iff
$E_i$ intersects $E_J$.

We shall need the following construction, see Figure~\ref{F101}.
Let $\sigma=\sigma_J$ be a face of $\D$ and 
$L\subset I$ the set of vertices of $\D$ contained in
$\sta_\D(\sigma)$. Thus $j\in L$ if and only if $E_j\cap E_J\ne\emptyset$.
Consider a rational point $v$ in the relative interior  of $\sigma$.
Given $0<\e<1$ rational and $j\in L$ set $e_j^\e:=\e e_j+(1-\e)v$. 
We shall define a projective simplicial subdivision $\D'=\D'(\e,v)$ of $\D$.

To define $\D'$, first consider a polyhedral subdivision $\D^\e=\D^\e(v)$ 
of $\D$ leaving the complement of $\sta_\D(\sigma)$ unchanged. 
The set of vertices of $\D^\e$ is precisely 
$(e_i)_{i\in I}\cup (e_j^\e)_{j\in L}$. The faces of 
$\D^\e$ contained in $\sta(\sigma)$ are of the following two types:
\begin{itemize}
\item
  if the convex hull $\Conv(e_{j_1}, \dots ,  e_{j_m})$ is a face of $\D$ 
  containing $\sigma$, then
  $\Conv( e_{j_1}^\e, \dots ,  e_{j_m}^\e)$ is a face of $\D^\e$;
\item
  if  $\Conv( e_{j_1}, \dots ,  e_{j_m})$ is a face of $\D$ contained in $\sta(\sigma)$ 
  but not containing $\sigma$, then both
  $\Conv( e_{j_1},\dots ,  e_{j_m})$ and $\Conv( e_{j_1}, \dots,e_{j_m},e_{j_1}^\e,\dots,e_{j_m}^\e)$ 
  are faces of $\D^\e$.
\end{itemize}
In a neighborhood of $v$, note that the subdivision $\D^\e$
is obtained by scaling $\D$ by a factor $\e$.
More precisely,  consider the affine map $\psi^\e:\sta(\sigma)\to\sta(\sigma)$ 
defined by $\psi^\e (w) = \e w+(1-\e)v$.
Then $\sigma^\e:=\psi^\e (\sigma)$ is the face of $\D^\e$ containing 
$v$ in its relative interior, and $\psi^\e(\sta_{\D}(\sigma))=\sta_{\D^\e} (\sigma^\e)$.
In particular, even though $\D^\e$ is not simplicial in general, 
all faces of $\D^\e$ containing $\sigma^\e$ are simplicial.

We claim that $\D^\e$ is projective. 
To see this, write $v=\sum_{j\in J}s_je_j$, with $s_j>0$ rational 
and $\sum s_j=1$. For $j\in J$, define a linear function $\lambda_j$ 
on $\sum_{i\in I}\R_+e_i\supset\D$ by $\lambda_j(\sum t_ie_i)=-t_j/s_j$
and set $h=\max\{\max_{j\in J}\lambda_j,-(1-\e)\}$. 
A suitable integer multiple of $h$ `
is then a strictly convex 
support function for $\D^\e$ in the sense of~\S\ref{S102}.

Now define $\D'=\D'(\e)$ as a simplicial subdivision of 
$\D^\e$ obtained using repeated barycentric subdivision
in a way that leaves $\sta_{\D^\e}(\sigma^\e)$ unchanged. 
By~\cite[pp.115--117]{KKMS}, $\D'$ is still projective.

Note that $\sigma':=\sigma^\e$ is the face of $\D'$ 
containing $v$ in its relative interior, 
For $j\in L$ set $e'_j=e^\e_j$. These are the vertices of $\D'$ contained
in $\sta_{\D'}(\sigma')$.
\begin{figure}[h]
\begin{center}
  \includegraphics[width=10cm]{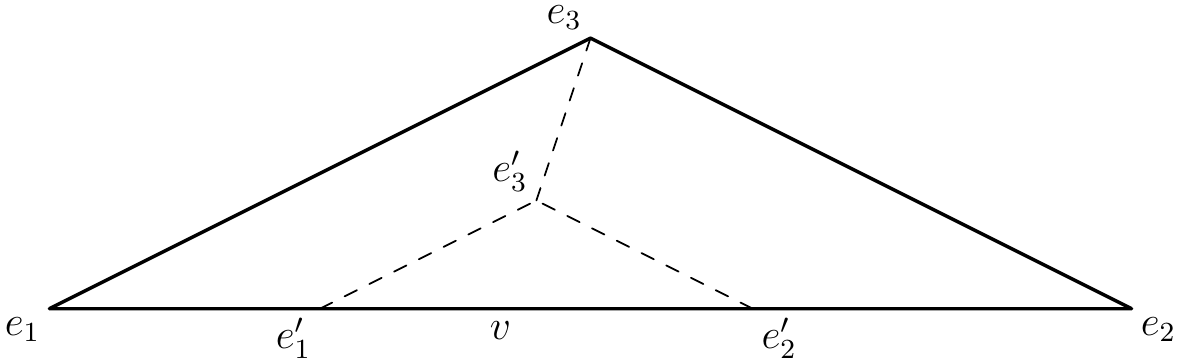}
    \end{center}
    \caption{The subdivision of~\S\ref{S103}.
      Here $v$ lies in the relative interior of the simplex $\sigma$ of $\D$
      with vertices $e_1$ and $e_2$. 
      The picture shows the intermediate subdivision $\D^\e$, where
      $v$ lies in the relative interior of the simplex $\sigma'$ with 
      vertices $e'_1$ and $e'_2$.
      The final subdivision $\D'$ is obtained from $\D^\e$ by barycentric
      subdivision of the quadrilaterals
      $\Conv(e_1,e_3,e'_1,e'_3)$ and $\Conv(e_2,e_3,e'_2,e'_3)$}\label{F101}
  \end{figure}
%
%
\subsection{Proof of Proposition~\ref{prop:keyprop}}
Let $I$ be the set of vertices in $\D$, let
$L\subset I$ be the set of vertices contained in $\sta_\D(\sigma)$
and $J\subset L$ the set of vertices of $\sigma$.
Thus $\sigma=\sigma_ J$. 
The irreducible subvariety $E_J$ has codimension $|J|=p\ge1$.

Consider the simplicial projective subdivision $\D'=\D'(\e)$ 
constructed in~\S\ref{S103}.
For $j\in L$, $e'_j:=\e e_j+(1-\e)v$ is a vertex of $\D'$.
Recall that $\sigma'=\sigma'_J$ is the face of $\D'$ containing 
  $v$ in its relative interior.
  Since $\chi|_\sigma$ is assumed affine in a neighborhood of 
  $v$,  we may choose $\e>0$ small enough that:
  \begin{itemize}
  \item 
    $\chi$ is affine on $\sigma'\subset\sigma$
  \item 
    $\chi$ is affine on each segment $[v,e'_j]$, $j\in L$. 
  \end{itemize}
  
  Let $\rho:X'\to X$ be the birational morphism corresponding to the 
  subdivision $\D'$ of $\D$ as in Theorem~\ref{T102}. 
  Note that $\rho$ induces a generically finite map $E'_J\to E_J$ of projective $k$-varieties.
  Indeed, $E_J$ (resp.\ $E'_J$) is the closure of the center of $v$
  on $X$ (resp.\ $X'$), and both have codimension $|J|=p$ 
  in view of Theorem~\ref{T102}. 
  
  The following result 
  allows us to ``linearize'' the problem under consideration:
  \begin{lem}\label{L101} 
    We have
    \begin{equation*}
      \rho^* \left( \chi(v) Z + \sum_{j\in L} D_v\chi (e_j) \,b_j E_j \right)\Bigg|_{E'_J}
     =\sum_{j\in L}\chi (e'_j) \,b'_j E'_j\Bigg|_{E'_J}
    \end{equation*}
    in $\Pic(E'_J)_\Q$, where we have set $b'_j:=\ord_{E'_j}(Z)$. 
  \end{lem}
  Grant this result for the moment. We then have 
  \begin{equation*}
    \rho^*G=\sum_{i\in I'}\chi(e'_i)b'_iE'_i+\tilde{G'},
 \end{equation*}
  where $\tilde{G'}$ is an effective $\Q$-Cartier divisor on $X'$ whose support
  does not contain any of the $E'_i$. 
  \begin{lem}\label{L102}
    The support of $\tilde{G}'$ does not contain $E'_J$.
    Hence $\tilde{G'}|_{E'_J}$ is effective.
  \end{lem}
  The proof is given below. Grant this result for the moment.
  Set $r=\deg(\rho|_{E'_J})$. Then we have
  \begin{align*}
    r(G\cdot M^{n-p-1}\cdot E_J)
    &=\rho^*G\cdot(\rho^*M)^{n-p-1}\cdot E'_J\\
    &=(\sum_{j\in L}\chi(e'_j)b'_jE'_j+\tilde{G'})\cdot(\rho^*M)^{n-p-1}\cdot E'_J\\
    &\ge(\sum_{j\in L}\chi(e'_j)b'_jE'_j)\cdot(\rho^*M)^{n-p-1}\cdot E'_J\\
    &=\rho^*\left( \chi(v)Z + \sum_{j\in L} D_v\chi (e_j) \,b_j
      E_j\right)\cdot(\rho^*M)^{n-p-1}\cdot E'_J\\
    &=r\left( \chi(v)Z + \sum_{j\in L} D_v\chi (e_j)
      \,b_jE_j\right)\cdot M^{n-p-1}\cdot E_J.
  \end{align*}
  Here the first and last inequality follow from the projection formula.
  The inequality follows from Lemma~\ref{L102}.
  The second to last equality is a consequence of Lemma~\ref{L101}.
  We see that 
  \begin{equation}\label{equ:lowerdv}
    \chi(v)(Z\cdot M^{n-p-1}\cdot E_J)
    +\sum_{j\in L} D_v\chi (e_j)b_j(E_j\cdot M^{n-p-1}\cdot E_J)
    \le\theta_G.
  \end{equation}
  By induction, the $C^{0,1}$-norm of $\chi|_\sigma$ is under control.
  Since $v$ belongs to $\sigma=\sigma_J$, this gives
  \begin{equation}\label{e103}
    \chi(v) + \max_{j\in J}|D_v\chi(e_j)|\le A\min_\D\chi+B\theta_G,
  \end{equation}
  which together with~\eqref{equ:lowerdv} yields an upper bound
  \begin{equation}\label{e104}
    \sum_{j\in L\setminus J} D_v \chi (e_j)\, b_j\left(E_J\cdot E_j\cdot M^{n-p-1}\right)
    \le A\min_\D\chi+B\theta_G.
  \end{equation}
  Now the fact that $\chi$ is nonnegative and concave shows that
  \begin{equation}\label{e105}
    \min_{j\in L\setminus J} D_v\chi(e_j)\ge\min_{j\in L\setminus J}(\chi(e_j)-\chi(v))
    \ge -\chi(v)
    \ge-A\min_\D\chi-B\theta_G,
  \end{equation}
  where the last inequality follows from the inductive assumption.
  Note that $E_j|_{E_J}$ is a non-zero effective divisor for 
  $j\in L\setminus J$. As a consequence
  \begin{equation*}
    (E_J\cdot E_j\cdot M^{n-p-1})>0
  \end{equation*}
  since $M$ is ample.
  The inequalities~\eqref{e104} and~\eqref{e105} therefore imply
  that 
  \begin{equation*}
    \max_{j\in L\setminus J}D_v\chi(e_j)\le A\min_\D\chi+B\theta_G,
  \end{equation*}
  which completes the proof, since $e=e_j$ for some $j\in L\setminus J$.
\begin{proof}[Proof of Lemma~\ref{L101}]
  We write $v = \sum_{j\in J} s_j e_j$ with $s_j>0$ rational and
  $\sum_{j\in J} s_j=1$. Set $s_i=0$ for $i\in I\setminus J$.
  For $i\in I$ let $\chi_i$ be the function on $\D$ 
  that is affine on each face of $\D$
  and satisfies $\chi_i(e_j) =\d_{ij}$ for all $j\in I$. 
  Since $e'_j=\e e_j+(1-\e) v$ for $j\in L$ we get:
  \begin{equation*}
    \chi_i(e'_j) =
    \begin{cases}
      \e+(1-\e) s_i   & \text{if $i= j\in J$}\\
      (1-\e) s_i & \text{if $i \ne j \in J$} \\
      \e & \text{if $i= j\in L\setminus J$}\\
      0 & \text{if $i\ne  j\in L\setminus J$}
    \end{cases} 
  \end{equation*}
  By Theorem~\ref{T102}, 
  $E'_j$ intersects $E'_J$ if and only if $j\in L$.
  We thus have
  \begin{equation*}
    \rho^*(b_i E_i)|_{E'_J}= \sum_{j\in L} \chi_i(e'_j)b'_j\, E'_j|_{E'_J}\ \text{for all $i\in I$}
  \end{equation*}
  and 
  \begin{equation*}
  \rho^*Z|_{E'_J}=\sum_{j\in L}b'_jE'_j|_{E'_J} 
  \end{equation*}
  in $\Pic(E'_J)_\Q$, where $b'_j=\ord_{E'_j}(I_Z)$. 
  
  Recall also that $\chi$ is affine on each segment $[v,e'_i]$, so that
  $D_v\chi(e_i) = \e^{-1}\left(\chi(e'_i) - \chi(v)\right)$
  for $i\in L$. 
  We can now compute in $\Pic(E'_J)_\Q$
  \begin{multline*}
    \left. \rho^* \left(\sum_{i\in L} D_v \chi(e_i)\, b_i E_i\right) \right|_{E'_J}
    =\sum_{i\in L} \e^{-1} \left(\chi(e'_i) - \chi(v)\right)\left(\sum_{j\in L}\chi_i(e'_j) \,b'_j\, E'_j |_{E'_J}\right)
    =\\=
    \sum_{i\in J} \e^{-1} (\chi(e'_i)-\chi(v)) \left(\e\, b'_i \,E'_i|_{E'_J}
      +s_i\sum_{j\in J} (1-\e)\,b'_j\, E'_j|_{E'_J}\right) +\\+
    \sum_{i\in L\setminus J}\e^{-1}(\chi(e'_i) - \chi(v))\, \e\, b'_i \,E'_i |_{E'_J}
    =\\
    = 
    \sum_{i\in L}(\chi(e'_i)-\chi(v))\,b'_i\, E'_i |_{E'_J} 
    +
    \e^{-1}(1-\e)\left(\sum_{i\in J} s_i (\chi(e'_i) - \chi(v))\right)\left(\sum_{j\in J}\,b'_j\, {E'_j}|_{E'_J}\right)
    \\
    =\sum_{j\in L}(\chi(e'_j)-\chi(v))\,b'_j E'_j\Bigg|_{E'_J}.
 \end{multline*}
  The last equality follows from the fact that $\chi $ is affine on the simplex 
  $\sigma'_J$ of $\D'$ so that
  $\sum_{i\in J} s_i\chi(e'_i)=\chi(v) = \sum_{i\in J} s_i \chi(v)$.
  This concludes the proof.
\end{proof}
\begin{proof}[Proof of Lemma~\ref{L102}]
  By assumption, the function $w\mapsto w(G)$ is 
  affine on the face $\sigma'=\sigma'_J$ of $\D'$.
  By Theorem~\ref{T102}, the same is true of the function
  $w\mapsto\sum_{i\in I'}\f(e'_i)b'_iw(E'_i)$.
  From this we see that the function $w\mapsto w(\tilde{G}')$ is also affine on
  $\sigma'$. But by construction, this function vanishes at the
  vertices of $\sigma'$ and hence is identically zero on $\sigma'$.
  This implies that the support of $\tilde{G}'$ does not contain $E'_J$,
  so that the $\Q$-Cartier divisor $\tilde{G}'|_{E'_J}$ is effective, as claimed.
\end{proof}
%
%
\section{Consequences of Theorem~B}\label{S121}
In this final section we prove the various consequences of Theorem~B,
namely Theorems~A and A', Izumi's Theorem (in characteristic zero)
and Corollaries~C,~D and~E.
%
%
\subsection{Order functions, integral closure and  Rees valuations}\label{S113}
Let us return to the situation in the beginning of the introduction.
Thus $k$ is an algebraically closed field,
$Y$ is a normal variety over $k$ 
and $0\in Y$ is a closed point.
We do not assume that $Y$ is smooth outside $0$.
Write $\fm_0$ for the maximal ideal of
the local ring $\cO_{Y,0}$ at $0$.

For any function $f\in\cO_{Y,0}$ define
\begin{equation}\label{e110}
  \ord_0(f):=\max\{j\ge 0\mid f\in\fm_0^j\}.
\end{equation}
When $0$ is a smooth point of $Y$, $\ord_0$ is a divisorial valuation, 
associated to the exceptional divisor of the blowup of $Y$ at 0.
In the singular case, however, $\ord_0$
may not be a valuation.
Indeed, the sequence
$(\ord_0(f^n))_{n\ge1}$  which is clearly
superadditive in the sense that 
\begin{equation}\label{e106}
  \ord_0(f^{n+n'})\ge\ord_0(f^n)+\ord_0(f^{n'})
\end{equation}
may fail to be additive, that is, strict inequality may hold
in~\eqref{e106} for certain $n$, $n'$.

To remedy this particular fact, one defines
\begin{equation*}
  \widehat{\ord_0}(f):=\lim_{n\to0}\frac1n\ord_0(f^n);
\end{equation*}
the limit exists as a standard consequence of~\eqref{e106}.
The function $\widehat{\ord_0}$ is a special case of a construction introduced by 
Samuel~\cite{Sam52} and later studied extensively by 
Rees, see~\cite{ReesBook} and also~\cite{HSBook,LT,Swa11}.

Recall that the \emph{integral closure} $\overline{\fb}$
of an ideal $\fb\subset\cO_{Y,0}$ is an ideal defined as 
the set of elements $f\in\cO_{Y,0}$ that satisfy an equation
\begin{equation*}
  f^n+a_1f^{n-1}+\dots+a_n=0,
\end{equation*}
with $n\ge 1$ and $a_i\in\fb^i$ for $1\le i\le n$.
The following result  is valid in 
a context far more general than what we state here, see~\cite[Theorem~4.13]{Hun92} or~\cite[Proposition~1.14]{LT} .
\begin{thm}\label{T107}
  There exists an integer $N$ such that
  \begin{equation}\label{e109}
    \overline{\fb^n}\subset\fb^{n-N}
  \end{equation}
  for any ideal $\fb\subset\cO_{Y,0}$
and any $n\ge N$.
\end{thm}
Let $\nu:Y^+\to Y$ be the normalized blowup of $\fm_0$ and write
\begin{equation*}
  \fm_0\cdot\cO_{Y^+}=\cO_{Y^+}(-\sum_{i=1}^kr_iE_i),
\end{equation*}
where the $E_i$ are prime Weil divisors on $Y^+$ and $r_i\in\Z_{>0}$.
For each $i$ we have a divisorial valuation $\ord_{E_i}$
on $\cO_{Y,0}$. We normalize these as follows.
\begin{defi}\label{defRees}
  The  divisorial valuations $w_1,\dots,w_k$ defined by 
  \begin{equation*}
    w_i
    :=\frac{\ord_{E_i}}{r_i}
    =\frac{\ord_{E_i}}{\ord_{E_i}(\fm_0)}
  \end{equation*}
  are called the \emph{Rees valuations} of $\fm_0$.
\end{defi}
\begin{thm}\label{T106}
  There exists an integer $N>0$
  such that the following conditions hold 
  for any function $f\in\cO_{Y,0}$ and any $n\ge 1$:
  \begin{itemize}
  \item[(i)]
    $\widehat{\ord_0}(f)=\min_iw_i(f)$;
  \item[(ii)]
    $f\in\overline{\fm_0^n}$ if and only if $\widehat{\ord_0}(f)\ge n$;
  \item[(iii)]
   $\ord_0(f)\le\widehat{\ord_0}(f)\le\ord_0(f)+N$;
  \item[(iv)]
   $\ord_0(f)\le\widehat{\ord_0}(f)\le(N+1)\ord_0(f)$.
 \end{itemize}
\end{thm}
\begin{proof}
Since $\nu$ is also the normalized blow-up of $\fm_0^n$ for any $n\ge 1$, we
have
 $$\overline{\fm_0^n}=\nu_*\cO_{Y^+}(-\sum_inr_iE_i) ,$$ 
  see~\cite[Proposition~9.6.6]{LazBook}.  Hence
  \begin{equation}\label{e108}
    f\in\overline{\fm_0^n}
    \quad\text{if and only if}\quad
    \min_iw_i(f)\ge n.
  \end{equation}
  We first prove~(i). Pick $\lambda\in\Q_{\ge0}$.
  If $\min_iw_i(f)\ge\lambda$, then for $p$ sufficiently divisible we have 
  \begin{equation*}
    f^p\in\overline{\fm_0^{p\lambda}}\subset\fm_0^{p\lambda-N}
  \end{equation*}
  by~\eqref{e108} and~\eqref{e109},
  respectively. This gives
  $\ord_0(f^p)\ge p\lambda-N$ and hence 
  $\widehat{\ord_0}(f)\ge\lambda$.
  On the other hand, suppose $\widehat{\ord_0}(f)\ge\lambda$ and pick
  $0<\mu<\lambda$. 
  For $p$ sufficiently divisible we then have
  $\ord_0(f^p)\ge\mu p$, so that
  $f^p\in\fm_0^{p\mu}\subset\overline{\fm_0^{p\mu}}$.
  Using~\eqref{e108} we get 
  $\min_iw_i(f)=p^{-1}\min_i(f^p)\ge p\mu$ and hence 
  $\min_iw_i(f)\ge\lambda$, proving~(i).

  Now~(ii) follows immediately from~(i) and from~\eqref{e108}.
  As for~(iii), the first inequality is obvious and the second 
  results from~(ii) and~\eqref{e109}.
  Finally,~(iv) is a direct consequence of~(iii) when 
  $\ord_0(f)\ge1$ and is trivial when $\ord_0(f)<1$
  since in this case $f\not\in\fm_0$ and 
  $\ord_0(f)=\widehat{\ord_0}(f)=0$.
\end{proof}
\begin{rmk}
 Theorem~\ref{T106} is a special case of the 
  \emph{strong valuation theorem} due to~Rees
  and is valid much more generally,
  see~\cite{HSBook,LT,ReesBook}.
  Our presentation follows~\cite[\S9.6.A]{LazBook}.
\end{rmk}
%
%

\subsection{Proof of Izumi's Theorem}\label{SpfIzu}
Let $v$ be any divisorial valuation of $k(Y)$ centered at 0.
We may assume that $v$ is normalized by $v(\fm_0)=1$.
It is then clear that $v\ge\ord_0$.

It remains to prove that there exists a constant $C>0$
such that $v(f)\le C\ord_0(f)$ for all $f\in\cO_{Y,0}$.
For this part, we assume that $k$ has characteristic zero.
Using Hironaka's theorem~\cite{Hir64} we can find 
a projective birational morphism $\pi:X\to Y$ with $X$ smooth such that
the scheme theoretic preimage
$Z:=\pi^{-1}(0)$ is a divisor (not necessarily reduced)
with simple normal crossing support such that any nonempty intersection of irreducible components of $Z$ is irreducible. Note that we do not assume that $\pi$ is an 
isomorphism outside $|Z|$. 
We may also assume that the center of $v$ has codimension $1$ so that 
$v$ is a vertex in the dual complex $\D=\D(X,Z)$ as in~\S\ref{S104}.

Given a function $f\in\cO_{Y,0}$ define a continuous function $\chi=\chi_f$ 
on $\Delta$ by 
\begin{equation*}
  \chi(v)=v(f).
\end{equation*}
It is clear that $\chi>0$ on $\Delta$ if $f\in\fm_0$ and 
$\chi\equiv0$ otherwise. 
Note that replacing $Y$ by a suitable affine neighborhood of $0$, we may view $f$ as a section of the trivial line bundle
$\cO_X$.

We can therefore apply~\eqref{e102}. We get that
$v(f) \le A\, \min_\D \chi_f$ for some constant $A> 0$ independent on $f$.
It remains to relate $\min_\D\chi_f$ to $\ord_0(f)$.
To this end, we first prove
\begin{lem}\label{L113}
  For any function $f\in\cO_{Y,0}$ we have 
  $\min_\D\chi_f=\widehat{\ord_0}(f)$.
\end{lem}
\begin{proof}
  If $\ord_0(f)\ge n$, then $f\in\fm_0^n$
  and hence $\min_\D\chi_f\ge\min_{v\in\D}v(\fm_0^n)=n$.
  Replacing $f$ by a power, we get
  $\min_\D\chi_f\ge\widehat{\ord_0}(f)$.
  
Since $Z$ is a divisor and $X$ is smooth, $\pi$ must factor through the 
normalized blowup $\nu:Y^+\to Y$ of $0$. This implies that 
all the Rees valuations of $0$ appear as (some of the) vertices
of the dual complex $\D$.  This observation and Theorem~\ref{T106}~(i)
now imply the reverse inequality. 
\end{proof}
Finally we have
$v(f) \le A \,\min_\D \chi_f  = A\, \widehat{\ord_0}(f)\le A (N+1)\, \ord_0(f)$
by Theorem~\ref{T106}~(i), and the proof of Izumi's theorem is complete.

\begin{rmk}
Observe that the proof does not rely on the Lipschitz estimates of Theorem~B, and follows from
a direct intersection theoretic computation which is similar to Izumi's original argument.
\end{rmk}

%
%
\subsection{Proof of Theorem~A}\label{S118}
Consider a projective birational morphism $\pi:X\to Y$ with $X$ smooth such that 
$Z:=\pi^{-1}(0)$ is a divisor with simple normal crossing support such that any nonempty intersection of irreducible components of $Z$ is irreducible.  
Let $\D=\D(X,Z)$ be the dual complex.

Given a function $f\in\cO_{Y,0}$,   the  function $\chi (v) := v(f)$
is continuous on $\D$.
It is clear that $\chi>0$ on $\Delta$ if $f\in\fm_0$ and 
$\chi\equiv0$ otherwise. 
As above, we may view $f$ as a section of the trivial line bundle
$\cO_X$. We can therefore apply Theorem~B and get that
$\chi_f$ is concave on each face and Lipschitz continuous with 
Lipschitz constant at most $A\min_\D\chi_f$. 

Thus the Lipschitz constant of $\chi_f$ is bounded
from above by at most $A(N+1)\ord_0(f)$
by Lemma~\ref{L113} and
Theorem~\ref{T106}~(iv), concluding the proof of Theorem~A.
%
%
\subsection{Mixed multiplicities}\label{S130}
Let $(Y,0)$ be as before.
The Hilbert-Samuel multiplicity  of an $\fm_0$-primary ideal $\fa\subset \cO_{Y,0}$ is defined as the limit 
$$
e(\fa) = \lim_{n\to\infty} \frac{m!}{n^m} \dim_k (\cO_{Y,0} /\fa^n)~.
$$
Recall that the mixed multiplicities of any two $\fm_0$-primary ideals $\fa_1,\fa_2$
are a sequence of $m+1$ integers
$e(\fa_1^{[0]};\fa_2^{[m]}), e(\fa_1^{[1]};\fa_2^{[m-1]}), ... , e(\fa_1^{[m]};\fa_2^{[0]})$
such that 
$$
e(\fa_1^r\cdot \fa_2^s) = \sum_{i=0}^m {m\choose i} \, e(\fa_1^{[m-i]};\fa_2^{[i]})\, r^{m-i}s^i
$$
for all $r,s\in\Z_+$, see~\cite[\S 2]{cycles evanescents} or~\cite[\S 1.6.B]{LazBook}.
Observe that mixed multiplicities are symmetric in their argument
$ e(\fa_1^{[m-i]};\fa_2^{[i]}) = e(\fa_2^{[i]}; \fa_1^{[m-i]})$.

These multiplicities also have the following geometric interpretation, see~\cite[\S 1.6.B]{LazBook}.
Let $\nu : Y^+ \to Y$ be any birational proper map that dominates the normalized blowups of $\fa_1$ and $\fa_2$. For $\e = 1,2$ write
$\fa_\e\cdot \cO_{Y^+} = \cO_{Y^+} ( - \sum_j r_{j,\e} E_j)$, with $r_{j,\e}\in \Z_{>0}$.
Then 
\begin{equation}\label{eq:geom}
e(\fa_1^{[m-i]};\fa_2^{[i]}) = - \left(\sum_j r_{j,1} E_j \right)^{m-i} \cdot \left(\sum_j r_{i,2} E_j \right)^{i}
\end{equation}
Since the antieffective divisors $-\sum_j r_{j,\e} E_j$ are $\pi$-exceptional and $\pi$-nef, 
it follows that mixed multiplicities are decreasing with respect to the inclusion of ideals: 
\begin{equation*}\label{eqinclu}
\fa_1 \subset \fa_1' \Rightarrow e(\fa_1^{[m-i]};\fa_2^{[i]})
\ge e(\fa'_1\!{}^{[m-i]};\fa_2^{[i]})
~.\tag{***}
\end{equation*}
Pick any  rank $1$ valuation $v$ on $\cO_{Y,0}$ centered at  $0$. Then  the sequence of valuation ideals $$\fa(v,n)=\{f\in\cO_{Y,0}\mid v(f)\ge n\}$$
forms a graded sequence in the sense that 
$$\fa(v,n)\cdot\fa (v,n')\subset \fa (v,n+n')$$ for any $n,n'$.
Recall that the volume of $v$ is defined by
$$
\vol(v) := \limsup_{n\to\infty}\frac{\dim_k(\cO_{Y,0}/\fa(v,n))}{n^m/m!} \in [0, +\infty).
$$
It is a theorem that the volume is actually defined as a limit, see~\cite{ELS,LazMus,cutkosky-mult}.
\begin{prop}\label{prop:defmixed}
For any integer $0\le i \le m$, 
the sequence $\frac1{n^i} e(\fa (v,n)^{ [i]}; \fm_0^{[m-i]}))$ converges to a
positive real number $\a_i(v)$.

We have $\a_0(v) = e(\fm)$ and $\a_m(v) = \vol(v)$.
Moreover these numbers satisfy the Teissier inequalities
\begin{equation}\label{eKT}
\a_i(v)^2 \le \a_{i-1}(v)\, \a_{i+1}(v), \, i= 1, ..., m-1 ~.
\end{equation}
\end{prop}
\begin{proof}
Fix $0\le i \le m$, and write $e_n := e(\fa (v,n)^{[i]}; \fm_0^{[m-i]}))$.
Since $\fa(v,1)^n\subset \fa(v,n)$, we have
$$e(\fa (v,n)^{[i]}; \fm_0^{[m-i]}))\le e(\fa (v,1)^n{}^{[i]}; \fm_0^{[m-i]})) = n^i \, e(\fa (v,1)^{[i]}; \fm_0^{[m-i]}))~,$$ 

It follows that  $\frac{e_n}{n^i} \le e_1$ is bounded from above. 
Pick $\e>0$ and choose $N$ such that 
$\frac{e_N}{N^i} \le \liminf_l \frac{e_l}{l^i} + \e$.
For any $n\ge N$ write $n = pN + q$ with $p,q \in \Z_+$ and $ 0 \le q \le N-1$.
Then 
$\fa (v,n) \supset \fa( v,N)^p\cdot \fa(v,q)\supset \fa( v,N)^{p+1}$ hence by monotonicity of mixed multiplicities
$$\frac{e_n}{n^i} \le \frac{e_N}{N^i} \,  \frac{N^i (p+1)^i}{n^i}
\le  \left( \liminf_l \frac{e_l}{l^i} + \e \right)\, \left(\frac{N (p+1)}{ pN + q}\right)^i~.$$
If $p$ is large enough, then we get $\frac{e_n}{n^i} \le (1-\e) ( \liminf_l \frac{e_l}{l^i} + \e)$
which implies $\frac{e_n}{n^i}\to \liminf_l \frac{e_l}{l^i}$. 

The fact that $\a_0(v) = e(\fm_0)$ follows from the definition, and the equality $\a_m(v) = \vol(v)$ is a theorem 
proved successively in greater generality in~\cite{ELS,Mus02,LazMus,cutkosky-mult}.
The inequalities~\eqref{eKT} follow from the usual Teissier inequalities for mixed multiplicities, see~\cite[Theorem 1.6.7 (iv)]{LazBook} and ultimately result from the Hodge index theorem.

Since $\a_i(v)$ is nonnegative, and $\a_0(v) >0$, the
Teissier inequalities imply that $\a_i(v) >0$  for all $i$. 
\end{proof}
The invariant $\a_1(v)$ is closely related to the optimal Izumi constant. F
The linking number~\cite{huckaba,samuel} of any two rank $1$ valuations on $\cO_{Y,0}$ centered at $0$
is defined by
$$
\b(v/w) := \sup_{f\in \fm_0 } \frac{v(f)}{w(f)} \in (0,\infty]~.
$$
By Izumi's theorem, this number is finite whenever $v$ and $w$ are both quasimonomial.

\begin{prop}\label{p8}
Let $w_i = \frac1{r_i}\ord_{E_i}$ be the Rees valuations normalized as in Definition~\ref{defRees}.
Then there exists integers $a_i\ge 1$ such that
\begin{equation}\label{eq666}
\a_1(v)  = \sum_i \frac{a_i}{r_i}  \, \b(v/w_i)^{-1}~.
\end{equation}
In particular, there exists a constant $C>0$ depending only on $(Y,0)$ but not
on $v$ such that  
$$
C^{-1}\, \a_1(v) \le \left(\sup_{\fm_0} \frac{v}{\ord_0}\right)^{-1}\le C\, \a_1(v)~.
$$
Finally, when  $\cO_{Y,0}$ admits a unique Rees valuation, there exists a positive rational
number $\theta$ such that
$$\a_1(v) =  \theta\, \left(\sup_{\fm_0} \frac{v}{\widehat{\ord_0}}\right)^{-1}$$
for all $v$; and $\theta =1$ when $0$ is a smooth  point.
\end{prop}
\begin{rmk}
Suppose $\dim(Y) =2$ and the point $0$ is smooth.
Then~\cite[Remark 3.33]{valtree}, Propositions~\ref{prop:defmixed} and~\ref{p8} 
imply
$$
\a_1(v) = \left(\sup_{\fm_0} \frac{v}{\ord_0}\right)^{-1}, \, \text{ and }
\a_2(v) v(\fm_0) = \left(\sup_{\fm_0} \frac{v}{\ord_0}\right)^{-1},
$$
It follows  that
$\a_1^2(v) = \a_{0}(v)\, \a_{2}(v)$ if and only if
$\sup_{\fm_0} v/\ord_0  = v(\fm_0)$. The latter condition is equivalent
to $v$ being proportional to $\ord_0$.
\end{rmk}

\begin{proof}
As in \S\ref{S113}, let $\nu : Y^+ \to Y$ be the normalized blowup of $\fm_0$, write
$\fm_0\cdot \cO_{Y^+} = \cO_{Y^+} ( - \sum_i r_i E_i)$, so that $w_i = r_i^{-1} \ord_{E_i}$.
\begin{lem}{~\!\cite[Lemma 2.4]{jonmus}}\label{lembdiv}
$\b(v/w)^{-1} = \lim_n \frac1n w (\fa(v,n)) $.
\end{lem} 
By~\eqref{eq:geom}, we get
$$
e (\fa (v,n)^{[1]}; \fm_0^{[m-1]})
= 
\left( \sum_i \ord_{E_i} (\fa(v,n))E_i \right) \cdot \left( - \sum_j r_j E_j\right)^{m-1}
= \sum \frac{a_i}{r_i} \, w_i (\fa(v,n))
$$
with $a_i := E_i \cdot ( \sum_j - r_j E_j)^{m-1}$. 
Then~\eqref{eq666} follows from the previous lemma, by
dividing by $n$ and letting $n\to \infty$.

\smallskip

Pick any rank $1$ valuation $v$ on $\cO_{Y,0}$ centered at $0$, and write
$$\beta(v) := \sup_{\fm_0} \frac{v}{\ord_0}~.$$ 
We have
$ v \le \beta(v)\, \ord_0 \le  \beta(v) w_i$ for all $i$, 
so that
$ \b(v/w_i)\le \beta(v)$, and
$\a_1(v) \ge (\sum \frac{a_i}{r_i}) \beta(v)^{-1}$.

Conversely, $\a_1(v) \le (\sum \frac{a_i}{ r_i}) \b(v/w_i)^{-1}$ for some $i$,
whereas $\b(v/w_i) \ge C^{-1} \sup_{\fm_0} v/ \ord_0 = C^{-1} \beta(v)$ by Izumi's theorem
applied to $w_i$.
This proves $\a_1(v) \le C(\sum \frac{a_i}{ r_i}) \beta(v)^{-1}$ as required.

\smallskip

Finally if $\cO_{Y,0}$ has a unique Rees valuation $w_i$, then~\eqref{eq666}
and Theorem~\ref{T106} (i) imply $\a_1(v) = \frac{a_i}{r_i} (\sup_{\fm_0} v/\widehat{\ord_0})^{-1}$. Finally when $0$ is a smooth point, it is easy to see
that $a_i= r_i = 1$. 
This concludes the proof.
\end{proof}

\begin{proof}[Proof of Lemma~\ref{lembdiv}]
We give a proof for completeness, see~\cite[Lemma 2.4]{jonmus}.
Observe first that since $\fa(v,n)$ is a graded sequence of ideals, then 
the limit $\lim \frac1n w (\fa(v,n))$ exists as $n\to\infty$.
Denote it by $\theta$.
For any $n$, we have
$n = v (\fa(v,n)) \le w(\fa(v,n)) \, \b(v/w)$ hence $1\le \theta \, \b(v/w)$.
Conversely, pick any $f\in \fm_0$, and let $n:= v(f)$. Then 
$f\in \fa(v,n)$ and  $w(\fa(v,n)) \le w(f)$ implies
$$
 \frac{1}{n} w(\fa(v,n)) \, v(f) \le\, w(f) 
$$
Replacing $f$ by $f^l$ and letting $l\to\infty$ we get
$v(f)\le \theta^{-1} w(f)$ which implies
$\b(v/w)\le  \theta^{-1}$. This concludes the proof.
\end{proof}

%
%
\subsection{Proof of Corollaries~D and~E}\label{S117}

We start by proving Corollary~D.
Fix $ 0\le i \le m$. 
For any $v\in\D$ we have $v\ge\ord_0$, so that 
$\fa(v,n)\supset\fm_{0}^n$ for any $n\ge 1$.
This implies that
\begin{equation}
  \a_i(v)\le C
\end{equation}
for all $v\in\D$.

Fix a metric on $\D$ compatible with the affine structure.
Fix $v\in\D$ and an integer $n\ge1$.
It follows from Theorem~A that the function
$w\mapsto w(\fa(v,n))$ is Lipschitz continuous with
Lipschitz constant at most 
\begin{equation*}
  A\ord_0(\fa(v,n))\le A\, v(\fa(v,n)).
\end{equation*}
Consider a valuation $w\in\D$ such that 
$\|v-w\|<\frac1A$. We then have
\begin{align*}
  w(\fa(v,n))
  &\ge v(\fa(v,n))-Av(\fa(v,n))\|v-w\|\\
  &=(1-A\|v-w\|)v(\fa(v,n))\\
  &\ge (1-A\|v-w\|)n,
\end{align*}
so that 
\begin{equation}\label{eqincl}
  \fa(v,n)\subset\fa(w,n(1-A\|v-w\|)).
\end{equation}
and by~\eqref{eqinclu}
\begin{align*}
  \frac{e(\fa(v,n)^{[i]},\fm_0^{[m-i]})}{n^i}
  &\ge
  \frac{e(\fa(w,n(1-A\|v-w\|))^{[i]},\fm_0^{[m-i]})}{n^i}\\
  &=
  (1-A\|v-w|)^i\, \frac{e(\fa(w,n(1-A\|v-w\|))^{[i]},\fm_0^{[m-i]})}{n^i(1-A\|v-w\|)^i}.
\end{align*}
Letting $n\to\infty$ we obtain
\begin{equation*}
  \a_i(v)\ge(1-A\|v-w\|)^i\a_i(w),
\end{equation*}
so that 
\begin{align*}
  \a_i(w)-\a_i(v)
  &\le(1-(1-A\|v-w\|)^i)\a_i(w)\\
  &\le C(1-(1-A\|v-w\|)^i)\\
  &\le iCA\|v-w\|,
\end{align*}
where we have used the inequality 
$1-(1-t)^i\le it$ for $0\le t\le 1$.
Exchanging the roles of $v$ and $w$ we 
conclude that 
\begin{equation*}
  |\a_i(v)-\a_i(w)|\le iCA\|v-w\|
\end{equation*}
for all $v,w\in\D$ such that $\|v-w\|<\frac1A$.
This completes the proof of Corollary~D.

\medskip

We now prove Corollary~E. 
Pick four valuations $v,v',w,w'\in \D$.
As in the proof of Corollary~D we may assume $\max \{ \| v - w\|, \| v' - w'\|\} \le A$.
By~\eqref{eqincl} we get
$$
\fa(w,n) \subset \fa( v, n ( 1- A \| v-w\|)) \text{ and } \fa(w',n)) \le \fa( v', n ( 1- A \| v'-w'\|))~.
$$
In particular $n \le w (\fa(w,n)) \le w (\fa( v, n ( 1- A \| v-w\|)))$ so that
$\b(v/w)^{-1}\ge (1- A \| v-w\|)^{-1}$  by Lemma~\ref{lembdiv}. Since
$$\b(v/v') \le \b(v/w)\,  \b(w/v')\le \b(v/w) \, \b(w'/v') \, \b(w/w'),$$ 
we conclude
$$
\b(v/v') \le (1- A \| v-w\|) ( 1- A \| v'-w'\|)\, \b(w/w')
$$
which implies the Lipschitz continuity of $(v,v') \mapsto \b(v/v')$ as above.

%
%
\subsection{Proof of Theorem~A'}\label{S116}
We keep the notation from the introduction.
Thus we let $(X,Z)$, $\D$, $J$, $\sigma_J$, $\xi_J$ and $(z_j)_{j\in J}$
be as in the discussion before Theorem~A'.
By Cohen's Theorem there is an isomorphism
$\widehat{\cO_{X,\xi_J}}\simeq\kappa[[\xi_j]]_{j\in J}$,
where $\kappa(\xi_J)$ is the residue field at $\xi_J$.
Fix such an isomorphism. Given $f\in\cO_{Y,0}$ we can then write
\begin{equation*}
  f=\sum_{\b\in\Z_{\ge0}^J}a_\b z^\b
\end{equation*}
with $a_\b\in\kappa(\xi_J)$.
The Newton polydron $\Nw(f,J)$ is then defined as
\begin{equation*}
  \Nw(f,J):=\Conv\left(\bigcup_{a_\b\ne0}(\b+\R_{\ge0}^J)\right)\subset\R_{\ge0}^J.
\end{equation*}
Let us give an alternative description of the Newton polyhedron, which 
shows that it does not depend on the choice of coordinates $z_j$ or
the choice of isomorphism in Cohen's theorem.
Consider $\sigma_J$ as embedded as the unit simplex in 
$\sum_{j\in J}\R e_j\simeq\R^J$
and let $\langle\cdot,\cdot\rangle$ be the standard scalar product on $\R^J$.
We then have
\begin{align*}
  \b\in\Nw(f,J)
  &\Leftrightarrow\langle v,\b\rangle \ge v(f)\ \text{for all $v\in\sigma_J$}\\
  &\Leftrightarrow\langle v,-\b\rangle \le\f(v)\ \text{for all $v\in\sigma_J$}\\
  &\Leftrightarrow-\b\in\Nw(\f),
\end{align*}
where $\Nw(\f)$ denotes the Newton polyhedron of the piecewise affine 
convex function $\f=\log|f|$ on the simplex $\sigma_J$, 
as defined in~\S\ref{S115}. 

Fix a norm on $\R^J$.
By Theorem~A, the Lipschitz constant of $\f$ on $\sigma$ is 
bounded by $A\ord_0(\f)$.
If $\b\in\R^J$ is an extremal point of $\Nw(f,J)$, then 
$-\b$ is an extremal point of $\Nw(\f)$ and we conclude from
Proposition~\ref{P114} that $\|\b\|\le AC\ord_0(f)$,
concluding the proof of Theorem~A'.
%
%
\subsection{Proof of Corollary~C}
As in the introduction, we fix an embedding
$\A^m\hookrightarrow\P^m$ and call a smooth projective variety
$X$ an admissible compactification 
of $\A^m$ if $X$ admits a birational morphism
$\pi:X\to\P^m$ that is an isomorphism above $\A^m$ and such that the 
divisor $Z:=\pi^{-1}(\P^m\setminus\A^m)$ has simple normal crossing
support and that any nonempty intersection between irreducible
components of $Z$ is irreducible. Note that $Z$ then has connected
support as a consequence of Zariski's Main Theorem. Further, $X$
contains $\A^m$ as a Zariski open subset. 
We view the elements of the dual complex $\D=\D(X,Z)$ as
valuations on $k(X)$ normalized on $v(Z)=1$.

Let $L_d=\pi^*\cO(d)$ for $d\ge1$ and let $G$ be the pullback to $X$
of the zero locus on $P$ on $\P^m$. 
Thus $L_d=\cO_X(G)=L_1^{\otimes d}$ and we have
  $v(G)=v(P)+d$
for every $v\in\D$.
In particular, the functions $v\mapsto v(P)$ and $v\mapsto\chi_G(v):=v(G)$ 
on $\D$ have the same Lipschitz constant.
Now $\min_{v\in\D}v(P)=-d$, with the minimum being obtained 
at the divisorial valuation corresponding to the divisor 
$\P^m\setminus\A^m$, so $\min_\D\chi_G=0$.
We therefore get from Theorem~B that the Lipschitz constant of 
$\chi_G$ is bounded from above by 
\begin{align*}
  B\max_{J\subset I}|(G\cdot M^{n-|J|-1}\cdot E_J)|
  &=B\max_{J\subset I}|(L_d\cdot M^{n-|J|-1}\cdot E_J)|\\
  &=Bd\max_{J\subset I}|(L_1\cdot M^{n-|J|-1}\cdot E_J)|,
\end{align*}
which completes the proof.
%
%

\end{document}